\documentclass[10pt,a4paper]{article}
\usepackage[T1]{fontenc}
\usepackage[a4paper,margin=2.35cm]{geometry}
\usepackage{microtype}
\usepackage{amsmath,amssymb,amsthm,mathtools}
\usepackage{xcolor}
\usepackage{enumitem}
\usepackage{url}
\usepackage{authblk}
\usepackage{titlesec}
\usepackage{aliascnt}
\usepackage[
  hyperfootnotes=false,
  colorlinks=true,
  linkcolor=blue!45!black,
  citecolor=blue!45!black,
  urlcolor=blue!45!black
]{hyperref}
\usepackage{bookmark}
\usepackage[nameinlink,capitalise,noabbrev]{cleveref}

\allowdisplaybreaks[1]
\setlist{itemsep=0.15em,topsep=0.25em,parsep=0pt}
\hypersetup{
  pdftitle={Sharp bounds for minimal dependencies of linear-form powers},
  pdfauthor={Heng Li and Xizhi Liu}
}

\numberwithin{equation}{section}

\theoremstyle{plain}
\newtheorem{theorem}{Theorem}[section]
\newaliascnt{lemma}{theorem}
\newtheorem{lemma}[lemma]{Lemma}
\aliascntresetthe{lemma}
\newaliascnt{proposition}{theorem}
\newtheorem{proposition}[proposition]{Proposition}
\aliascntresetthe{proposition}
\newaliascnt{corollary}{theorem}
\newtheorem{corollary}[corollary]{Corollary}
\aliascntresetthe{corollary}
\newaliascnt{conjecture}{theorem}

\aliascntresetthe{conjecture}
\newaliascnt{claim}{theorem}

\aliascntresetthe{claim}

\theoremstyle{definition}
\newaliascnt{definition}{theorem}
\newtheorem{definition}[definition]{Definition}
\aliascntresetthe{definition}
\newaliascnt{problem}{theorem}
\newtheorem{problem}[problem]{Problem}
\aliascntresetthe{problem}
\newaliascnt{example}{theorem}

\aliascntresetthe{example}

\theoremstyle{remark}
\newaliascnt{remark}{theorem}
\newtheorem{remark}[remark]{Remark}
\aliascntresetthe{remark}
\newaliascnt{observation}{theorem}

\aliascntresetthe{observation}
\newaliascnt{fact}{theorem}

\aliascntresetthe{fact}

\crefname{definition}{Definition}{Definitions}
\crefname{theorem}{Theorem}{Theorems}
\crefname{lemma}{Lemma}{Lemmas}
\crefname{proposition}{Proposition}{Propositions}
\crefname{corollary}{Corollary}{Corollaries}
\crefname{conjecture}{Conjecture}{Conjectures}
\crefname{claim}{Claim}{Claims}
\crefname{problem}{Problem}{Problems}
\crefname{remark}{Remark}{Remarks}
\crefname{observation}{Observation}{Observations}
\crefname{fact}{Fact}{Facts}
\crefname{example}{Example}{Examples}
\crefname{enumi}{part}{parts}

\newcommand{\K}{K}

\newcommand{\PP}{\mathbb P}
\newcommand{\one}{\mathbf{1}}
\newcommand{\St}{\operatorname{St}}

\newcommand{\Span}{\operatorname{span}}
\newcommand{\Sym}{\operatorname{Sym}}

\titleformat{\section}{\normalfont\large\bfseries}{\thesection}{0.75em}{}
\titleformat{\subsection}{\normalfont\normalsize\bfseries}{\thesubsection}{0.75em}{}
\titleformat{\subsubsection}{\normalfont\normalsize\itshape}{\thesubsubsection}{0.75em}{}
\titlespacing*{\section}{0pt}{0.95\baselineskip}{0.45\baselineskip}
\titlespacing*{\subsection}{0pt}{0.75\baselineskip}{0.3\baselineskip}
\titlespacing*{\subsubsection}{0pt}{0.6\baselineskip}{0.25\baselineskip}

\makeatletter
\let\oldthebibliography\thebibliography
\renewcommand{\thebibliography}[1]{%
  \oldthebibliography{#1}%
  \setlength{\itemsep}{0.15em}%
  \setlength{\parskip}{0pt}%
}
\newcommand{\plainfootnotetext}[1]{%
  \begingroup
  \renewcommand\thefootnote{}%
  \long\def\@makefntext##1{\noindent ##1}%
  \footnotetext{#1}%
  \addtocounter{footnote}{-1}%
  \endgroup
}
\makeatother

\begin{document}

\title{\bf\Large Sharp bounds for minimal dependencies of linear-form powers}
\date{\today}
\author[1,2]{Heng~Li}
\author[3]{Xizhi~Liu}
\affil[1]{School of Mathematics, Shandong University, Jinan, China}
\affil[2]{Extremal Combinatorics and Probability Group, Institute for Basic Science, Daejeon, South Korea}
\affil[3]{School of Mathematical Sciences, University of Science and Technology of China, Hefei, China}

\maketitle

\begin{abstract}
Motivated by the dimension-bound part of a problem of Bukh, we study Veronese circuits: how large can the span of $t$ linear forms be if their $m$-th powers are minimally linearly dependent?
We prove the sharp finite dimension bound
\[
  \dim L\leq \frac{t+m-2}{m}.
\]
Here $\ell_1,\ldots,\ell_t$ are nonzero homogeneous linear forms over a field of characteristic zero, the powers $\ell_1^m,\ldots,\ell_t^m$ form a circuit, and $L=\Span\{\ell_1,\ldots,\ell_t\}$.
Rational-normal-curve configurations attain equality for infinitely many pairs $(t,\dim L)$; in particular, the affine bound itself is sharp and the optimal leading constant in Bukh's question is $1/m$.

The proof uses a coding-theoretic translation: the coefficient row space of the powers is the $m$-th Schur power of the coefficient code, and the minimality hypothesis makes this Schur power a full-support hyperplane to which the Schur-product Kneser theorem of Mirandola and Z\'emor applies.
The same method yields flat-concentration and interpolation criteria, a Cayley--Bacharach lower bound, Segre--Veronese and positive-characteristic variants, and Hilbert-function constraints for equality and near-equality in Veronese circuits.
\end{abstract}

\plainfootnotetext{\textit{Email:} \texttt{heng.li@sdu.edu.cn}, \texttt{liuxizhi@ustc.edu.cn}.}

\section{Introduction}

Throughout the paper, a linear form is a homogeneous polynomial of degree one.  Thus, after choosing coordinates $x_1,\ldots,x_n$ on a vector space over a field $\K$, a linear form has the shape
\[
  a_1 x_1+\cdots+a_n x_n
\]
with coefficients $a_i\in\K$ and with no constant term.  Equivalently, it is a linear functional on the underlying vector space.  We write $[t]=\{1,\ldots,t\}$.

Powers of linear forms are the basic rank-one objects in the symmetric tensor, or polynomial Waring, setting.  A decomposition
\[
  F=\lambda_1\ell_1^m+\cdots+\lambda_t\ell_t^m
\]
of a homogeneous form $F$ is a symmetric power-sum decomposition; geometrically, the points $[\ell_i^m]$ lie on the degree-$m$ Veronese variety.  The dimensions of secant varieties of Veronese varieties and the corresponding generic Waring ranks have a large classical literature, including the interpolation theorem of Alexander and Hirschowitz \cite{AlexanderHirschowitz}.  Related background is provided by Iarrobino and Kanev \cite{IarrobinoKanev} on power sums and Gorenstein algebras, and by modern accounts of tensor and secant geometry by Landsberg \cite{Landsberg} and Bernardi, Carlini, Catalisano, Gimigliano, and Oneto \cite{BernardiEtAl}.  Representations of forms as sums of powers of linear forms were studied by Reznick \cite{ReznickEvenPowers} and by Bia{\l}ynicki-Birula and Schinzel \cite{BialynickiBirulaSchinzel}; zero-relation and dependence patterns among polynomial powers were considered by Reznick \cite{ReznickPatterns}, and the linear-form case was studied by S{\l}adek \cite{Sladek}.

The present paper studies a different, more combinatorial, question.  We are not asking for the rank of a general form or the dimension of a secant variety.  Instead, we ask how large the span of the underlying linear forms can be when the corresponding points on the Veronese variety satisfy a support-minimal linear relation.  Thus the relevant object is a \emph{circuit} in the linear matroid represented by the Veronese images.

This circuit viewpoint is also naturally tied to interpolation.  If $X\subseteq\PP^n(\K)$ is a finite set of $\K$-rational points, then, after dualizing, the rank of the degree-$m$ evaluation map on $X$ is the rank of the vectors representing the Veronese images $\nu_m(X)$.  Therefore a circuit after the degree-$m$ Veronese embedding is exactly a minimal obstruction to independent interpolation in degree $m$.  This connection explains why Cayley--Bacharach conditions and Hilbert functions of finite point sets appear later in the paper; for background on these themes, see Davis--Geramita--Orecchia \cite{DGO}, Eisenbud--Green--Harris \cite{EGH}, Geramita--Kreuzer--Robbiano \cite{GKR}, and Kreuzer--Long--Robbiano \cite{KLR}.

We are concerned with support-minimal linear dependencies among powers of linear forms.  Given nonzero linear forms $\ell_1,\ldots,\ell_t$ and an integer $m\geq1$, suppose that
\(
  \ell_1^m,\ldots,\ell_t^m
\)
are linearly dependent, but that every proper subfamily is linearly independent.  We use \emph{circuit} in the linear-matroid sense: the whole family is dependent and every proper subfamily is independent.  The dimension-bound problem addressed here is to bound the dimension of
\[
  L=\Span\{\ell_1,\ldots,\ell_t\}
\]
in terms of the number $t$ of terms in the minimal dependence.

Bukh \cite{BukhProblems} isolated the problem.  In related work on extremal graphs without exponentially small bicliques, Bukh \cite{BukhBicliques} constructs auxiliary varieties with strong independence properties; a key algebraic input is to control minimal linear dependencies
\(
  \sum_{i=1}^t \alpha_i\ell_i^m=0
\)
among $m$-th powers of linear forms.  The earlier bounds used there build on ideas related to representations of polynomials by sums of powers of linear forms, in particular the work of Bia{\l}ynicki-Birula and Schinzel \cite{BialynickiBirulaSchinzel}.  Bukh then asked for the optimal dependence of $\dim L$ on $t$.

More precisely, Bukh \cite{BukhProblems} noted that a bound of the form $\dim L\leq c_m t+O_m(1)$ is known with $c_m=3/(m+4)$ for $m\geq2$, and asked whether the leading constant can be improved to $1/m$ in the following finite form.

\begin{problem}[Bukh \cite{BukhProblems}]\label{prob:bukh}
Let $m\geq2$.  Determine whether there exists a constant $C_m$ such that the following holds.
Let $\K$ be a field of characteristic zero, and let $\ell_1,\ldots,\ell_t$ be nonzero homogeneous linear forms over $\K$ whose $m$-th powers are linearly dependent, but every proper subfamily of these powers is linearly independent.
If
\(
  L=\Span\{\ell_1,\ldots,\ell_t\},
\)
then
\(
  \dim L\leq t/m+C_m?
\)
\end{problem}

The following theorem resolves the dimension-bound part of this problem by giving a sharp affine estimate.

\begin{theorem}\label{thm:main}
Let $\K$ be a field of characteristic zero.  Let $m\geq1$, and let $\ell_1,\ldots,\ell_t$ be nonzero homogeneous linear forms in finitely many variables over $\K$.  Suppose that $\ell_1^m,\ldots,\ell_t^m$ are linearly dependent, but every proper subfamily of these $m$-th powers is linearly independent.  Put
\(
  L=\Span\{\ell_1,\ldots,\ell_t\}.
\)
Then
\[
  \dim L\leq \frac{t+m-2}{m}.
\]
Equivalently, since $\dim L$ is an integer,
\[
  \dim L\leq \left\lfloor\frac{t+m-2}{m}\right\rfloor;
\]
in particular, $\dim L\leq t/m+O_m(1)$.
\end{theorem}

The bound in Theorem~\ref{thm:main} is attained by a uniform family of extremal configurations.  For $m=1$, this is the elementary statement that a circuit of $t$ vectors has rank $t-1$.  For $m\geq2$, the bound is best possible in the following strong sense: for every $r\geq1$ there are
\(
  t=m(r-1)+2
\)
linear forms whose $m$-th powers form a minimal dependence and whose span has dimension
\[
  r=\frac{t+m-2}{m}.
\]
These examples come from points on a rational normal curve; the construction is given in Section~\ref{sec:basic-sharpness}.  Consequently, the optimal leading constant in Bukh's dimension-bound question is $c_m=1/m$.  Moreover, because equality holds for the full affine expression when $t=m(r-1)+2$, the right-hand side of Theorem~\ref{thm:main} cannot be decreased by any positive constant independent of $t$ on these parameter values.

For later use, we record an equivalent projective form of Theorem~\ref{thm:main}.  If $W$ is a finite-dimensional vector space over $\K$, the degree-$m$ Veronese embedding
\[
  \nu_m\colon \PP(W)\to \PP(\Sym^m W),
  \qquad [\ell]\mapsto [\ell^m]
\]
sends a projective point represented by a linear form to the projective class of its $m$-th power.  Thus a support-minimal dependence among $\ell_i^m$ is the same as a circuit among the Veronese images $\nu_m([\ell_i])$.  Theorem~\ref{thm:main} is therefore equivalent to the following projective statement.

\begin{theorem}\label{thm:veronese-circuit-form}
Let $\K$ have characteristic zero.  If $p_1,\ldots,p_t\in\PP(W)$ form a circuit after applying $\nu_m$, and if
\(
  d=\dim\langle p_1,\ldots,p_t\rangle,
\)
then
\(
  t\geq md+2.
\)
\end{theorem}

This formulation also makes the characteristic restrictions explicit.  The characteristic-zero assumption in Theorem~\ref{thm:main} can be replaced by $\operatorname{char}\K>m$; see Corollary~\ref{cor:large-characteristic-main}.  In small positive characteristic, Frobenius effects lead to a digit-sum version of the bound, with the exponent $m$ replaced by the base-$p$ digit sum $w_p(m)$; see Theorem~\ref{thm:frobenius-corrected}.

The proof of Theorem~\ref{thm:main} uses Schur powers of linear codes.  Products and powers of codes under coordinatewise multiplication have been systematically developed in coding theory, with connections to algebraic geometry, secret sharing, and multiplication algorithms; see Randriambololona \cite{Randriambololona} and the subsequent work of Mirandola and Z\'emor \cite{MirandolaZemor}.  Our use of Schur products is somewhat different from their usual coding-theoretic use: the code records the coefficients of the linear forms, and its Schur powers record the coefficient spaces of the corresponding powers of forms.

More concretely, to $\ell_1,\ldots,\ell_t$ we associate their coefficient row space $C\subseteq\K^t$.  The coefficient row space of $\ell_1^m,\ldots,\ell_t^m$ is the $m$-th Schur power $C^{\langle m\rangle}$.  The support-minimal dependence makes this Schur power a full-support hyperplane.  The relevant lower bound for Schur products is a linear-code analogue of a theorem of Kneser \cite{Kneser} in additive combinatorics: the Schur-product Kneser theorem of Mirandola and Z\'emor \cite[Theorem~3.3]{MirandolaZemor}.  In the present setting, the full-support hyperplane forces the relevant stabilizer to be trivial, so Kneser growth can be iterated through the intermediate Schur powers.

At the level of dimensions, the central estimate is
\[
  t-1=\dim C^{\langle m\rangle}
  \geq m\dim C-(m-1).
\]
Since $\dim C=\dim L$, this is exactly the desired bound.  Sections~\ref{sec:kneser} and~\ref{sec:proof-main} make this reduction precise.

The paper is organized as follows.  Section~\ref{sec:applications} formulates consequences and extensions of the main theorem.  Section~\ref{sec:basic-sharpness} proves the projective reformulation and constructs the rational-normal-curve examples showing sharpness.  Section~\ref{sec:kneser} isolates the Schur-product Kneser input, and Section~\ref{sec:proof-main} proves Theorem~\ref{thm:main}.  The results stated in Section~\ref{sec:applications} are proved in Sections~\ref{sec:geometric-consequences}--\ref{sec:extremal-stability}: Section~\ref{sec:geometric-consequences} gives the flat-concentration, interpolation, and Cayley--Bacharach consequences; Section~\ref{sec:product-positive-characteristic} treats Segre--Veronese and positive-characteristic variants; and Section~\ref{sec:extremal-stability} analyzes equality and near-equality through Hilbert functions.

\section{Applications and further results}\label{sec:applications}

This section formulates the consequences of Theorem~\ref{thm:main} and of the Schur-product method.  The first group of statements is incidence-theoretic: a Veronese dependence contains a circuit whose underlying points span a flat of controlled dimension, and this flat-concentration principle gives interpolation and Cayley--Bacharach lower bounds.  The second group extends the argument to Segre--Veronese maps and to positive characteristic; in the latter case Frobenius twists replace the exponent $m$ by the base-$p$ digit sum $w_p(m)$.  The final group concerns the sharp and near-sharp regimes, where the same mechanism imposes Hilbert-function constraints and a Cayley--Bacharach duality bound.  The proofs of these statements are given after the proof of Theorem~\ref{thm:main}, in the same order as the subsections below.

\subsection{Flat concentration, interpolation, and Cayley--Bacharach sets}

The first applications translate Veronese dependence into incidence-geometric constraints.  They show that a dependence after the degree-$m$ Veronese embedding must already be witnessed by many points in a low-dimensional projective flat.  In this subsection, a flat means a nonempty projective linear subspace defined over $\K$, and all dimensions of flats are projective dimensions.

\begin{proposition}\label{prop:flat-cluster-extraction}
Let $\K$ have characteristic zero, and let $X\subseteq\PP^n(\K)$ be a finite set.  If $\nu_m(X)$ is linearly dependent, then there is a subset $S\subseteq X$ such that $\nu_m(S)$ is a circuit and 
\(
  |S|\geq me+2
\)
with $e=\dim\langle S\rangle$. 
Consequently, any flat-sparsity condition of the form
\[
  |X\cap\Pi|\leq m\dim\Pi+1
\]
for all nonempty projective $\K$-flats $\Pi$ rules out all degree-$m$ Veronese dependencies in $X$.
\end{proposition}

As a first consequence, the flat-sparsity criterion implies the following interpolation statement.  To state it, we use the usual Hilbert function of a finite set of points.

Let $X\subseteq\PP^n(\K)$ be a finite reduced set of $\K$-rational points.  We write
\[
  H_X(j)=\dim_\K \operatorname{im}\left(
  H^0(\PP^n,\mathcal O_{\PP^n}(j))\longrightarrow \K^X
  \right)
\]
for its degree-$j$ Hilbert function, after choosing arbitrary nonzero homogeneous representatives for the points of $X$.

\begin{corollary}\label{cor:hilbert-flat}
Let $\K$ have characteristic zero, and let $X\subseteq\PP^n(\K)$ be finite.  Suppose that every nonempty projective $\K$-flat $\Pi\subseteq\PP^n$ satisfies
\[
  |X\cap\Pi|\leq m\dim\Pi+1,
\]
where $\dim\Pi$ denotes projective dimension.  Then
\(
  H_X(m)=|X|.
\)
Equivalently, the points of $X$ impose independent linear conditions on homogeneous forms of degree $m$.
\end{corollary}

The same minimal-obstruction argument yields a lower bound for sets with the Cayley--Bacharach property.  We recall the definition before stating the result.

\begin{definition}\label{def:cbm}
A finite reduced set $X\subseteq\PP^n(\K)$ of $\K$-rational points has the \emph{Cayley--Bacharach property in degree $m$}, abbreviated $CB(m)$, if for every $x\in X$, every homogeneous degree-$m$ form vanishing on $X\setminus\{x\}$ also vanishes at $x$.
\end{definition}

For sets minimal with respect to this property, the flat-concentration principle gives the same lower bound as for Veronese circuits.

\begin{corollary}\label{cor:cb-lower-bound}
Let $\K$ have characteristic zero.  Let $X\subseteq\PP^n(\K)$ be a nonempty finite reduced set of $\K$-rational points which is inclusion-minimal among nonempty finite reduced subsets with the $CB(m)$ property.  If $d=\dim\langle X\rangle$, then
\(
  |X|\geq md+2.
\)
The bound is attained by the rational-normal-curve examples in Section~\ref{sec:basic-sharpness}.
\end{corollary}

\subsection{Segre--Veronese and positive-characteristic variants}

We next record extensions illustrating the flexibility of the Schur-product method.  The same argument applies to partially symmetric rank-one tensors and to Frobenius-decomposed powers in characteristic $p$.

\begin{theorem}\label{thm:segre-veronese}
Assume that either $\operatorname{char}\K=0$ or $\operatorname{char}\K>\max_q m_q$.  Let $U_1,\ldots,U_s$ be finite-dimensional vector spaces over $\K$, and let $m_1,\ldots,m_s\geq1$.  For each $i\in[t]$ and $q\in[s]$, let $\ell_{i,q}\in U_q\setminus\{0\}$, and put
\[
  F_i=\ell_{i,1}^{m_1}\otimes\cdots\otimes \ell_{i,s}^{m_s}
  \in\bigotimes_{q=1}^s\Sym^{m_q}U_q.
\]
Suppose that $F_1,\ldots,F_t$ are linearly dependent, but every proper subfamily is linearly independent.  If
\[
  L_q=\Span\{\ell_{1,q},\ldots,\ell_{t,q}\}\subseteq U_q,
\]
then
\[
  t\geq 2+\sum_{q=1}^s m_q(\dim L_q-1).
\]
\end{theorem}

In positive characteristic, the Schur-product argument naturally produces a bound in terms of the sum of the base-$p$ digits of the exponent.

For the positive-characteristic results, write the base-$p$ expansion of $m$ as
\[
  m=m_0+m_1p+\cdots+m_ap^a,
  \qquad 0\leq m_j<p,
\]
and set $w_p(m)=m_0+m_1+\cdots+m_a$.

\begin{theorem}\label{thm:frobenius-corrected}
Let $\K$ be a field of characteristic $p>0$.  Let $m\geq1$, and let $\ell_1,\ldots,\ell_t$ be nonzero homogeneous linear forms over $\K$.  Suppose that $\ell_1^m,\ldots,\ell_t^m$ are linearly dependent, but every proper subfamily is linearly independent.  Put
\(
  L=\Span\{\ell_1,\ldots,\ell_t\}.
\)
Then
\[
  \dim L\leq \frac{t+w_p(m)-2}{w_p(m)}.
\]
Equivalently, if $d=\dim L-1$, then $t\geq w_p(m)d+2$.
\end{theorem}

For $p>m$, this digit sum is equal to $m$, and the characteristic-zero projective bound is recovered verbatim.

\begin{corollary}\label{cor:veronese-circuit-large-characteristic}
Let $\K$ have characteristic $p>m$, let $W$ be a finite-dimensional vector space over $\K$, and let $\nu_m$ be the degree-$m$ Veronese embedding.  If points $p_1,\ldots,p_t\in\PP(W)$ form a circuit after applying $\nu_m$, and if
\(
  d=\dim\langle p_1,\ldots,p_t\rangle,
\)
then
\(
  t\geq md+2.
\)
\end{corollary}

The following example shows that the digit-sum correction is necessary in the pure Frobenius case.

\begin{proposition}\label{prop:frobenius-degeneration}
Let $\K$ be a field of characteristic $p>0$, and let $m=p^a$.  For every $r\geq1$, there are $t=r+1$ nonzero linear forms $\ell_1,\ldots,\ell_t$ spanning an $r$-dimensional space such that $\ell_1^m,\ldots,\ell_t^m$ form a circuit.
\end{proposition}

\subsection{Extremal and near-extremal Veronese circuits}

We finally state the structural consequences for equality and near-equality in the Veronese-circuit bound.  Throughout this subsection, $\K$ has characteristic zero, $m\geq1$, and $X\subseteq\PP^a(\K)$ is a reduced nondegenerate finite set of $\K$-rational points.  We write
\[
  \Delta H_X(j)=H_X(j)-H_X(j-1),\qquad H_X(-1)=0.
\]

\begin{theorem}\label{thm:extremal-cb}
Let $\K$ be a field of characteristic zero.  Let $m\geq1$ and $a\geq1$, and let $X\subseteq\PP^a(\K)$ be a reduced nondegenerate finite set of $\K$-rational points.  Then $\nu_m(X)$ is a circuit and
\(
  |X|=ma+2
\)
if and only if $X$ has $CB(m)$ and
\[
  \Delta H_X=(1,\underbrace{a,\ldots,a}_{m\text{ times}},1).
\]
Equivalently, the nonzero first differences are
\(
  1,\underbrace{a,\ldots,a}_{m\text{ times}},1.
\)
\end{theorem}

Over an algebraically closed field, the same Hilbert-function condition has the following standard arithmetically Gorenstein interpretation.  This interpretation is not used in the proof of the main theorem or in the stability results.

\begin{remark}[Arithmetically Gorenstein interpretation]\label{rem:ag-interpretation}
Let $\K$ be an algebraically closed field of characteristic zero, and let $m\geq1$ and $a\geq1$.  For a reduced nondegenerate finite set $X\subseteq\PP^a(\K)$ of $\K$-rational points, the following conditions are equivalent:
\begin{enumerate}[label=\textup{(\roman*)}]
  \item $\nu_m(X)$ is a circuit and $|X|=ma+2$;
  \item $X$ has $CB(m)$ and $\Delta H_X=(1,\underbrace{a,\ldots,a}_{m\text{ times}},1)$;
  \item $X$ is arithmetically Gorenstein with $h$-vector $(1,\underbrace{a,\ldots,a}_{m\text{ times}},1)$.
\end{enumerate}
\end{remark}

We next turn to near-equality.  The following parameters measure the surplus of the Hilbert function above the extremal lower envelope.

Suppose now that $\nu_m(X)$ is a circuit and
\(
  |X|=ma+2+s
\)
for some $s\geq0$.  Define
\[
  \delta_j=H_X(j)-(ja+1)\qquad(0\leq j\leq m).
\]
For $m=1$, the theorem below forces $s=0$.

\begin{theorem}\label{thm:hilbert-stability}
Let $\K$ be a field of characteristic zero.  Let $m\geq1$ and $a\geq1$, and let $X\subseteq\PP^a(\K)$ be a reduced nondegenerate finite set of $\K$-rational points.  Assume that $\nu_m(X)$ is a circuit and
\(
  |X|=ma+2+s
\)
for an integer $s\geq0$.  Then
\[
  0=\delta_0=\delta_1\leq\delta_2\leq\cdots\leq\delta_m=s.
\]
If $m=1$, then $s=0$ and
\(
  \Delta H_X=(1,a,1).
\)
If $m\geq2$, equivalently,
\[
  \Delta H_X=(1,a,a+\varepsilon_2,\ldots,a+\varepsilon_m,1),
\]
where
\(
  \varepsilon_j\geq0
  \) and \(
  \sum_{j=2}^m\varepsilon_j=s.
\)
\end{theorem}

The circuit relation further imposes a duality constraint on these surplus terms, giving a second form of stability.

Let $E_j\subseteq\K^X$ be the degree-$j$ evaluation space.  Since $\nu_m(X)$ is a circuit, there is a vector $\lambda=(\lambda_x)_{x\in X}\in(\K^\times)^X$ such that $E_m=\lambda^\perp$.  Define
\[
  \langle u,v\rangle_\lambda=\sum_{x\in X}\lambda_x u_x v_x.
\]

\begin{theorem}\label{thm:duality-stability}
Under the hypotheses of Theorem~\ref{thm:hilbert-stability}, for every $0\leq i\leq m$ one has
\(
  \delta_i+\delta_{m-i}\leq s.
\)
Equivalently,
\[
  H_X(i)+H_X(m-i)\leq |X|.
\]
Moreover, if $D_i=E_i^\perp/E_{m-i}$, where $E_i^\perp$ is taken with respect to $\langle-,-\rangle_\lambda$, then
\[
  \dim_\K D_i=s-\delta_i-\delta_{m-i}
  \qquad(0\leq i\leq m),
\]
and hence
\[
  \sum_{i=0}^m\dim_\K D_i
  =(m+1)s-2\sum_{i=0}^m\delta_i
  \leq (m+1)s.
\]
\end{theorem}

When the surplus is one, the monotonicity and duality constraints determine the first difference completely.

\begin{corollary}\label{cor:s-one}
Under the hypotheses of Theorem~\ref{thm:hilbert-stability}, assume $m\geq2$ and $s=1$, equivalently
\(
  |X|=ma+3.
\)
Then there exists an integer $k$ with $k>m/2$ such that
\[
  \delta_j=
  \begin{cases}
  0,&j<k,\\
  1,&j\geq k,
  \end{cases}
\]
for $0\leq j\leq m$.  Equivalently,
\[
  \Delta H_X=(1,a,\ldots,a,a+1,a,\ldots,a,1),
\]
with the unique surplus $a+1$ occurring in degree $k>m/2$.
\end{corollary}

\section{Veronese circuits and sharpness}\label{sec:basic-sharpness}

We first record a base-change fact that will be used both in the sharpness discussion and later in the large-characteristic form of the theorem.

\begin{lemma}\label{lem:base-change-circuits}
Let $E/\K$ be a field extension, and let $v_1,\ldots,v_t$ be vectors in a finite-dimensional $\K$-vector space $V$.  Then the rank of the family over $\K$ is equal to the rank of the scalar-extended family in $V\otimes_\K E$ over $E$.  Consequently, linear independence, linear dependence, and the circuit property are preserved by field extension.
\end{lemma}

\begin{proof}
Choose a basis of $V$ and form the corresponding matrix with columns $v_i$.  The rank is the largest size of a nonzero minor.  Since a minor with entries in $\K$ is zero over $E$ if and only if it is zero over $\K$, the rank is unchanged after scalar extension.  The statements about independence, dependence, and circuits follow by applying this rank equality to all subfamilies.
\end{proof}

\subsection{The projective formulation}

\begin{proof}[Proof of Theorem~\ref{thm:veronese-circuit-form}]
Choose representatives $p_i=[\ell_i]$.  Then
\[
  d+1=\dim\Span\{\ell_1,\ldots,\ell_t\}.
\]
Theorem~\ref{thm:main} gives
\[
  d+1\leq \frac{t+m-2}{m},
\]
which is equivalent to $t\geq md+2$.
\end{proof}

\subsection{Rational-normal-curve extremal examples}

We now show, over any field $\K$ of characteristic zero, that equality in the affine bound of Theorem~\ref{thm:main} is attained for infinitely many parameter values.  In particular, the leading constant $1/m$ cannot be improved.

Fix $r\geq1$.  If $r=1$, the equality example is the two-form family $\ell_1=\ell_2=x_0$.  Thus assume $r\geq2$ and put
\[
  M=m(r-1),
  \qquad
  t=M+2=m(r-1)+2.
\]
Choose pairwise distinct scalars $a_1,\ldots,a_t\in\K$.  This is possible because every field of characteristic zero is infinite.  In variables $x_0,\ldots,x_{r-1}$, define
\[
  \ell_i=x_0+a_i x_1+a_i^2 x_2+\cdots+a_i^{r-1} x_{r-1}.
\]
Then $\Span\{\ell_1,\ldots,\ell_t\}$ has dimension $r$, because the coefficient matrix is Vandermonde of rank $r$.

Now consider $\ell_i^m$.  The coefficient of $x_0^{\alpha_0}x_1^{\alpha_1}\cdots x_{r-1}^{\alpha_{r-1}}$ in $\ell_i^m$ is
\[
  \binom{m}{\alpha_0,\ldots,\alpha_{r-1}}
  a_i^{\alpha_1+2\alpha_2+\cdots+(r-1)\alpha_{r-1}}.
\]
Thus each coefficient of $\ell_i^m$, as a function of $a_i$, is a scalar multiple of a monomial $a_i^s$ with $0\leq s\leq m(r-1)=M$.  Hence the coefficient row space of $\ell_1^m,\ldots,\ell_t^m$ is contained in the $(M+1)$-dimensional space spanned by
\[
  (1,\ldots,1),\quad (a_i)_{i=1}^t,\quad (a_i^2)_{i=1}^t,\quad \ldots,\quad (a_i^M)_{i=1}^t.
\]
Since $t=M+2$, the $t$ forms $\ell_1^m,\ldots,\ell_t^m$ are linearly dependent.

It remains to check minimality.  For every integer $0\leq s\leq M$, one can write $s=j_1+\cdots+j_m$ with $0\leq j_\nu\leq r-1$.  Indeed, write $s=q(r-1)+b$ with $0\leq b<r-1$.  If $b=0$, use $q$ summands equal to $r-1$ and the remaining summands equal to $0$.  If $b>0$, then $q\leq m-1$ because $s\leq m(r-1)$; use $q$ summands equal to $r-1$, one summand equal to $b$, and the remaining summands equal to $0$.
Taking $\alpha_j=\#\{\nu:j_\nu=j\}$ gives a monomial of total degree $m$ whose coefficient row is a nonzero scalar multiple of $(a_i^s)_{i=1}^t$.  This scalar is nonzero because $\operatorname{char}\K=0$.  Therefore the coefficient matrix contains the full Vandermonde block
\(
  (a_i^s)_{0\leq s\leq M,\ 1\leq i\leq t}.
\)
Thus any choice of $M+1=t-1$ columns gives a square Vandermonde matrix with distinct parameters $a_i$, and hence is nonsingular.  Hence any proper subfamily of size $t-1$ among $\ell_1^m,\ldots,\ell_t^m$ is linearly independent.  A fortiori every smaller proper subfamily is linearly independent.

Thus we have constructed a minimally dependent family with $\dim L=r$ and $t=m(r-1)+2$.  Consequently,
\[
  r=\frac{t+m-2}{m},
\]
so equality holds in Theorem~\ref{thm:main} for infinitely many values of $t$.  Since equality holds for $t=m(r-1)+2$, the affine bound \((t+m-2)/m\) cannot be decreased on this infinite congruence class.  In particular, the right-hand side of Theorem~\ref{thm:main} cannot be lowered by any fixed positive constant, and no leading constant smaller than $1/m$ can hold in general.

\begin{corollary}\label{cor:constant}
The optimal leading constant in Theorem~\ref{thm:main} is $c_m=1/m$.
\end{corollary}

\begin{proof}
Theorem~\ref{thm:main} gives a bound of the form
\[
  \dim L\leq \frac{1}{m}t+O_m(1),
\]
so $c_m=1/m$ is admissible.  Conversely, the rational-normal-curve construction above gives minimally dependent families with $t=m(r-1)+2$ and $\dim L=r$.  Hence
\[
  \frac{\dim L}{t}=\frac{r}{m(r-1)+2}\longrightarrow \frac1m
  \qquad\text{as }r\to\infty.
\]
No smaller leading constant can therefore hold for all such families.
\end{proof}

\section{The Schur-product Kneser input}\label{sec:kneser}

Throughout this section, $\K$ is an infinite field.  For vectors in $\K^t$, multiplication is coordinatewise.  If $S,T\subseteq \K^t$ are linear subspaces, their product is
\[
  ST=\Span\{s*t:s\in S,\ t\in T\}.
\]
For a subspace $V\subseteq \K^t$, define
\[
  \St(V)=\{u\in\K^t:uV\subseteq V\}.
\]
The scalar line $\K\one$, where $\one=(1,\ldots,1)$, is always contained in $\St(V)$.  We say that $V\subseteq\K^t$ has \emph{full support} if, for every coordinate $i$, some vector in $V$ has nonzero $i$-th coordinate.  A vector $x\in\K^t$ is \emph{invertible} if all its coordinates are nonzero.  Let $V^\times$ denote the set of invertible vectors in $V$.

The result below is the trivial-stabilizer case of the Schur-product Kneser theorem of Mirandola and Z\'emor \cite[Theorem~3.3]{MirandolaZemor}, who prove the more general inequality
\[
  \dim ST\geq \dim S+\dim T-\dim\St(ST)
\]
for coordinatewise products of nonzero subspaces of $\K^t$.  We include a proof of the special case needed here.

\begin{lemma}\label{lem:invertible-basis}
If $V\subseteq\K^t$ has full support, then $V$ has a basis consisting of invertible vectors.
\end{lemma}

\begin{proof}
Choose a basis $v_1,\ldots,v_k$ of $V$.  For $\alpha\in\K$, put
\[
  y_\alpha=v_1+\alpha v_2+\cdots+\alpha^{k-1}v_k.
\]
For every coordinate $j$, the $j$-th coordinate of $y_\alpha$ is a polynomial in $\alpha$.  Since $V$ has full support, this polynomial is not identically zero.  The bad values of $\alpha$ are therefore contained in a finite union of proper zero sets of nonzero univariate polynomials.  Because $\K$ is infinite, there are infinitely many $\alpha$ for which $y_\alpha$ is invertible.  Choosing $k$ distinct such values of $\alpha$ gives a basis, since the change-of-basis matrix is Vandermonde.
\end{proof}

\begin{lemma}\label{lem:finitely-many-subalgebras}
There are only finitely many unital subalgebras of $\K^t$ under coordinatewise multiplication.
\end{lemma}

\begin{proof}
Let $H\subseteq\K^t$ be a unital subalgebra.  Define an equivalence relation on $\{1,\ldots,t\}$ by
\[
  i\sim j
  \quad\Longleftrightarrow\quad
  h_i=h_j\text{ for every }h\in H.
\]
Then $H$ is contained in the space of vectors that are constant on the blocks of this partition.

Conversely, fix a block $B$ of the partition.  For each other block $B'$, by definition of the partition there is some $h_{B'}\in H$ whose value on $B$ differs from its value on $B'$.  Here $h(B)$ denotes the common value of $h$ on the block $B$.  The vector
\[
  \prod_{B'\neq B}
  \frac{h_{B'}-h_{B'}(B')\one}{h_{B'}(B)-h_{B'}(B')}
\]
belongs to $H$, is equal to $1$ on $B$, and is equal to $0$ on every other block.  Thus $H$ is exactly the algebra of vectors constant on the blocks of its partition.  Since there are only finitely many partitions of $\{1,\ldots,t\}$, there are only finitely many such algebras.
\end{proof}

\begin{lemma}\label{lem:e-transform}
Let $S,T\subseteq\K^t$ be nonzero subspaces.  Suppose that $T$ has a basis of invertible vectors.  For every $x\in S^\times$ there exist a unital subalgebra $H_x\subseteq\K^t$ and a subspace $V_x\subseteq\K^t$ such that
\[
  H_xV_x=V_x,
  \qquad
  xT\subseteq V_x\subseteq ST,
\]
and
\[
  \dim V_x+\dim H_x\geq \dim S+\dim T.
\]
\end{lemma}

\begin{proof}
First reduce to the case $x=\one\in S$.  Put $S'=x^{-1}S$.  If the case $\one\in S'$ gives a unital subalgebra $H$ and a subspace $V'$ with
\[
  HV'=V',
  \qquad
  T\subseteq V'\subseteq S'T,
\]
and
\[
  \dim V'+\dim H\geq \dim S+\dim T,
\]
then $V_x=xV'$ satisfies
\[
  HV_x=V_x,
  \qquad
  xT\subseteq V_x\subseteq ST,
  \qquad
  \dim V_x+\dim H\geq \dim S+\dim T.
\]
Thus it suffices to prove the case $x=\one\in S$.

We prove the case $x=\one$ with $\one\in S$ by induction on $k=\dim S$.  If $k=1$, then $S=\K\one$.  Take $H_x=\K\one$ and $V_x=T$.

Assume $k>1$.  For an invertible vector $e\in T$, define
\[
  S(e)=S\cap Te^{-1},
  \qquad
  T(e)=T+Se.
\]
Since $\one\in S$ and $e\in T$, we have $\one\in S(e)$, so $S(e)\neq0$.  To see the dimension formula, consider the linear map
\[
  \phi_e\colon S\to\K^t/T,
  \qquad
  s\mapsto se+T.
\]
Then $\ker\phi_e=S\cap Te^{-1}=S(e)$, and $T(e)/T\simeq\operatorname{im}\phi_e$.  Hence
\[
  \dim T(e)=\dim T+\dim S-\dim S(e),
\]
or equivalently $\dim S(e)+\dim T(e)=\dim S+\dim T$.  Also $S(e)T(e)\subseteq ST$.  Indeed, $S(e)T\subseteq ST$, and if $a\in S(e)$, then $a=be^{-1}$ for some $b\in T$, so for $s\in S$ we have $a(se)=bs\in ST$.

If $S(e)=S$ for every $e\in T^\times$, then $S\subseteq Te^{-1}$, and therefore $Se\subseteq T$ for every invertible $e\in T$.  Since the invertible vectors in $T$ span $T$, it follows that $ST\subseteq T$.  As $\one\in S$, also $T\subseteq ST$, hence $ST=T$.  Let $H$ be the unital subalgebra generated by $S$, and let $V=T$.  Since $ST=T$, repeated multiplication by elements of $S$ preserves $T$, and hence $HV=V$.  Also $T\subseteq V\subseteq ST$, and
\[
  \dim V+\dim H\geq \dim T+\dim S.
\]

Otherwise, choose $e\in T^\times$ with $0<\dim S(e)<\dim S$.  Since $T$ has a basis of invertible vectors, $T$ has full support.  As $T\subseteq T(e)$, the space $T(e)$ also has full support, and by Lemma~\ref{lem:invertible-basis} it has a basis of invertible vectors.  Applying the induction hypothesis to the pair $S(e),T(e)$ yields a unital subalgebra $H$ and a subspace $V$ with
\[
  HV=V,
  \qquad
  T\subseteq T(e)\subseteq V\subseteq S(e)T(e)\subseteq ST,
\]
and
\[
  \dim V+\dim H\geq \dim S(e)+\dim T(e)=\dim S+\dim T.
\]
This completes the induction.
\end{proof}

\begin{theorem}[Schur-product Kneser input]\label{thm:schur-kneser}
Let $\K$ be an infinite field.  Let $S,T\subseteq\K^t$ be nonzero subspaces.  If $\St(ST)=\K\one$, then $\dim ST\geq \dim S+\dim T-1$.
\end{theorem}

\begin{proof}
Since $\St(ST)=\K\one$, the product $ST$ has full support; otherwise arbitrary values on the unused coordinates would give a larger stabilizer.  Hence both $S$ and $T$ have full support.  By Lemma~\ref{lem:invertible-basis}, $T$ has a basis of invertible vectors, and $S$ has infinitely many invertible vectors.

Let $k=\dim S$.  Choose a basis $s_1,\ldots,s_k$ of $S$ and put
\[
  y_\alpha=s_1+\alpha s_2+\cdots+\alpha^{k-1}s_k.
\]
As in Lemma~\ref{lem:invertible-basis}, there are infinitely many $\alpha$ such that $y_\alpha$ is invertible.  For each such $y_\alpha$, Lemma~\ref{lem:e-transform} gives a unital subalgebra $H_{y_\alpha}$ and a space $V_{y_\alpha}$.

By Lemma~\ref{lem:finitely-many-subalgebras}, there are only finitely many possible subalgebras.  Hence there exist $k$ distinct scalars $\alpha_1,\ldots,\alpha_k$ such that $y_{\alpha_1},\ldots,y_{\alpha_k}$ are invertible, form a basis of $S$, and have the same associated algebra $H$.  Write $x_i=y_{\alpha_i}$ and $V_i=V_{x_i}$.

Since $x_1,\ldots,x_k$ is a basis of $S$, we have $ST=x_1T+\cdots+x_kT$.  As $x_iT\subseteq V_i\subseteq ST$ for every $i$, it follows that $ST=V_1+\cdots+V_k$.  Since each $V_i$ is $H$-stable and $ST=V_1+\cdots+V_k$, the space $ST$ is $H$-stable.  Hence $H\subseteq\St(ST)=\K\one$.

Thus $\dim H=1$.  For any $i$, Lemma~\ref{lem:e-transform} gives
\[
  \dim ST+1\geq \dim V_i+\dim H\geq \dim S+\dim T.
\]
Therefore
\[
  \dim ST\geq \dim S+\dim T-1.
\]
\end{proof}

\section{Proof of the main theorem}\label{sec:proof-main}

We first translate the problem into Schur powers.  Let $r=\dim L$.  Choose a basis $x_1,\ldots,x_r$ for $L$.  Write
\[
  \ell_i=a_{1i} x_1+\cdots+a_{ri} x_r,
  \qquad i=1,\ldots,t.
\]
Let $A=(a_{ji})$ be the corresponding $r\times t$ matrix, and let $C\subseteq\K^t$ be the row space of $A$.  Then $\dim C=r$.

For a subspace $C\subseteq\K^t$, define its Schur powers by
\[
  C^{\langle1\rangle}=C,
  \qquad
  C^{\langle j+1\rangle}=CC^{\langle j\rangle}.
\]
We shall use the elementary identity
\(
  C^{\langle a\rangle}C^{\langle b\rangle}=C^{\langle a+b\rangle},
\)
which follows directly from associativity and commutativity of coordinatewise multiplication.

\begin{lemma}\label{lem:rank-schur-power}
Fix $j\geq1$.  Assume either $\operatorname{char}\K=0$ or $\operatorname{char}\K>j$.  Then the rank of the family $\ell_1^j,\ldots,\ell_t^j$ equals $\dim C^{\langle j\rangle}$.
\end{lemma}

\begin{proof}
The rank of the family $\ell_1^j,\ldots,\ell_t^j$ equals the dimension of the row space of their coefficient matrix in the monomial basis of homogeneous degree-$j$ polynomials in $x_1,\ldots,x_r$.  Indeed, since $x_1,\ldots,x_r$ are linearly independent linear forms, they may be extended to a coordinate system on the ambient dual space.  In particular, the degree-$j$ monomials in $x_1,\ldots,x_r$ are linearly independent.

Expand
\[
  \ell_i^j=(a_{1i} x_1+\cdots+a_{ri} x_r)^j.
\]
For a multi-index $\alpha=(\alpha_1,\ldots,\alpha_r)$ with $|\alpha|=j$, the coefficient of $x_1^{\alpha_1}\cdots x_r^{\alpha_r}$, as $i$ varies, is
\[
  \binom{j}{\alpha}(a_{1i}^{\alpha_1}\cdots a_{ri}^{\alpha_r})_{i=1}^t,
\]
that is, a nonzero scalar multiple of a coordinatewise product of $j$ rows of $A$.  Conversely, every product of $j$ basis rows of $C$ appears in this way, up to a nonzero multinomial factor.  Since every element of $C$ is a linear combination of the rows of $A$, the coordinatewise product of any $j$ elements of $C$ is, by multilinearity, a linear combination of coordinatewise products of $j$ basis rows of $A$.  Under the stated characteristic hypothesis, these multinomial factors do not vanish.  Hence the coefficient row space of $\ell_1^j,\ldots,\ell_t^j$ is precisely $C^{\langle j\rangle}$.
\end{proof}

For Theorem~\ref{thm:main}, the characteristic-zero hypothesis ensures Lemma~\ref{lem:rank-schur-power} for $j=m$.  The same proof gives the large-characteristic form in Corollary~\ref{cor:large-characteristic-main}.  In small positive characteristic, Section~\ref{sec:product-positive-characteristic} replaces this lemma by a Frobenius-decomposed coefficient-space computation.

Because $\ell_1^m,\ldots,\ell_t^m$ are minimally linearly dependent, their rank is $t-1$.  By Lemma~\ref{lem:rank-schur-power},
\[
  \dim C^{\langle m\rangle}=t-1.
\]
Moreover, the linear relation is unique up to scalars and has full support: if a relation had some zero coefficient, then the remaining proper subfamily of $\ell_1^m,\ldots,\ell_t^m$ would already be linearly dependent, contradicting minimality.  Thus there exist nonzero scalars $\lambda_1,\ldots,\lambda_t\in\K^\times$ such that
\[
  \lambda_1\ell_1^m+\cdots+\lambda_t\ell_t^m=0.
\]
The vector $\lambda=(\lambda_1,\ldots,\lambda_t)$ annihilates the coefficient row space of $\ell_1^m,\ldots,\ell_t^m$, which is $C^{\langle m\rangle}$.  Thus
\[
  C^{\langle m\rangle}
  =\left\{z=(z_1,\ldots,z_t)\in\K^t:\sum_{i=1}^t\lambda_i z_i=0\right\}.
\]
So $C^{\langle m\rangle}$ is a hyperplane whose normal vector has no zero coordinate.  Since every $\lambda_i\neq0$, for each coordinate $j$ one can choose $z$ in this hyperplane with $z_j\neq0$; hence it has full support.

We next record the stabilizer property needed to apply Theorem~\ref{thm:schur-kneser}.

\begin{lemma}\label{lem:hyperplane-stabilizer}
Let
\[
  H=\left\{z\in\K^t:\sum_{i=1}^t\lambda_i z_i=0\right\},
\]
where each $\lambda_i\neq0$.  Then $\St(H)=\K\one$.
\end{lemma}

\begin{proof}
Let $u=(u_1,\ldots,u_t)\in\St(H)$.  Then for every $z\in H$,
\(
  \sum_{i=1}^t\lambda_i u_i z_i=0.
\)
Thus the linear functional $z\longmapsto\sum_{i=1}^t\lambda_i u_i z_i$ vanishes on the hyperplane
\[
  H=\ker\left(z\longmapsto\sum_{i=1}^t\lambda_i z_i\right).
\]
Since both linear functionals vanish on the same codimension-one subspace $H$, they are scalar multiples of each other.  Hence there exists $c\in\K$ such that $\lambda_i u_i=c\lambda_i$ for all $i$.  Since $\lambda_i\neq0$, we obtain $u_i=c$ for all $i$.  Hence $u\in\K\one$.
\end{proof}

\begin{lemma}\label{lem:all-stabilizers-trivial}
For every $1\leq j\leq m$, $\St(C^{\langle j\rangle})=\K\one$.
\end{lemma}

\begin{proof}
For $j=m$, this follows from Lemma~\ref{lem:hyperplane-stabilizer}, because $C^{\langle m\rangle}$ is the hyperplane $\sum_{i=1}^t\lambda_i z_i=0$ with all $\lambda_i\neq0$.

Now let $1\leq j<m$, and suppose that $u\in\St(C^{\langle j\rangle})$.  Then $uC^{\langle j\rangle}\subseteq C^{\langle j\rangle}$.  By the Schur-power identity,
\[
  C^{\langle j\rangle}C^{\langle m-j\rangle}=C^{\langle m\rangle}.
\]
Therefore
\[
  uC^{\langle m\rangle}
  =uC^{\langle j\rangle}C^{\langle m-j\rangle}
  \subseteq C^{\langle j\rangle}C^{\langle m-j\rangle}
  =C^{\langle m\rangle}.
\]
Thus $u\in\St(C^{\langle m\rangle})=\K\one$.  Since $\K\one\subseteq\St(C^{\langle j\rangle})$ is automatic, equality follows.
\end{proof}

\begin{proof}[Proof of Theorem~\ref{thm:main}]
Let $d_j=\dim C^{\langle j\rangle}$.  Then $d_1=\dim C=r$.
Because each $\ell_i$ is nonzero, the coefficient code $C$ has full support: in every coordinate, some row of $A$ is nonzero.  Since $\K$ has characteristic zero, it is infinite, and Lemma~\ref{lem:invertible-basis} gives an invertible vector $c\in C$.  For every $j\geq1$, the coordinatewise product of $j$ copies of $c$ belongs to $C^{\langle j\rangle}$ and is invertible, so all Schur powers used below are nonzero.

For $1\leq j<m$, apply Theorem~\ref{thm:schur-kneser} to $S=C$ and $T=C^{\langle j\rangle}$.  Since $ST=C^{\langle j+1\rangle}$ and, by Lemma~\ref{lem:all-stabilizers-trivial}, $\St(C^{\langle j+1\rangle})=\K\one$, we have
\[
  d_{j+1}=\dim C^{\langle j+1\rangle}
  \geq \dim C+\dim C^{\langle j\rangle}-1
  =r+d_j-1.
\]
Iterating this inequality from $d_1=r$ gives
\[
  d_j\geq r+(j-1)(r-1)=jr-(j-1)
\]
for all $1\leq j\leq m$.  Taking $j=m$ gives $d_m\geq mr-(m-1)$.  But the support-minimal dependence gives $d_m=t-1$.  Therefore $t-1\geq mr-(m-1)$, or equivalently,
\[
  r\leq \frac{t+m-2}{m}.
\]
Since $r=\dim L$, Theorem~\ref{thm:main} follows.
\end{proof}

\begin{corollary}\label{cor:large-characteristic-main}
Let $\K$ be a field, let $m\geq1$, and assume either $\operatorname{char}\K=0$ or $\operatorname{char}\K>m$.  Let $\ell_1,\ldots,\ell_t$ be nonzero homogeneous linear forms over $\K$.  Suppose that $\ell_1^m,\ldots,\ell_t^m$ are linearly dependent, but every proper subfamily of these $m$-th powers is linearly independent.  Put
\(
  L=\Span\{\ell_1,\ldots,\ell_t\}.
\)
Then
\[
  \dim L\leq \frac{t+m-2}{m}.
\]
Equivalently, in projective form, every degree-$m$ Veronese circuit spanning a projective $d$-flat has at least $md+2$ points.
\end{corollary}

\begin{proof}
If $\operatorname{char}\K=0$, this is Theorem~\ref{thm:main}.  Suppose $\operatorname{char}\K>m$.  By Lemma~\ref{lem:base-change-circuits}, we may extend scalars to an algebraic closure; the circuit property and the dimension of $L$ are unchanged.  Over the resulting infinite field, Lemma~\ref{lem:rank-schur-power} applies for $j=m$, and the proof of Theorem~\ref{thm:main} goes through verbatim.
\end{proof}

\section{Proofs of the geometric applications}\label{sec:geometric-consequences}

We prove the geometric applications stated in Section~\ref{sec:applications}.  In each result the same three-step mechanism appears: pass to a support-minimal obstruction, apply the Veronese circuit bound, and interpret the resulting inequality as a flat-density certificate.

\begin{proof}[Proof of Proposition~\ref{prop:flat-cluster-extraction}]
Choose a minimal dependent subfamily of the vectors $\nu_m(X)$ and let $S$ be its support.  Then $\nu_m(S)$ is a circuit.  Applying Theorem~\ref{thm:veronese-circuit-form} gives $|S|\geq m\dim\langle S\rangle+2$.  The final statement follows because the flat $\langle S\rangle$ would otherwise contain at least $me+2$ points of $X$.
\end{proof}

\subsection{Hilbert functions and interpolation}

Let $X\subseteq\PP^n(\K)$ be a finite reduced set of $\K$-rational points.  We write
\[
  H_X(j)=\dim_\K \operatorname{im}\left(
  H^0(\PP^n,\mathcal O_{\PP^n}(j))\longrightarrow \K^X
  \right)
\]
for its degree-$j$ Hilbert function, after choosing arbitrary nonzero homogeneous representatives for the points of $X$.  Changing representatives rescales coordinates and does not change the rank.

\begin{proof}[Proof of Corollary~\ref{cor:hilbert-flat}]
If the degree-$m$ Veronese images of the points of $X$ were dependent, choose a minimally dependent subfamily.  If its original points span a projective $d$-flat, Theorem~\ref{thm:veronese-circuit-form} gives at least $md+2$ points in that flat, contradicting the hypothesis.
\end{proof}

\subsection{Cayley--Bacharach lower bounds}

Recall Definition~\ref{def:cbm}.

\begin{proof}[Proof of Corollary~\ref{cor:cb-lower-bound}]
For $x\in X$, the $CB(m)$ property says that the degree-$m$ evaluation functional at $x$ lies in the span of the evaluation functionals at the other points of $X$.  Equivalently, every Veronese vector $\nu_m(x)$ lies in the span of the other Veronese vectors.  Thus the vector configuration $\nu_m(X)$ is dependent.

Choose an inclusion-minimal dependent subset $C\subseteq X$ for this vector configuration.  Then $\nu_m(C)$ is a circuit.  A circuit has the property that every one of its elements lies in the span of the others, so $C$ has $CB(m)$.  By the inclusion-minimality of $X$ among nonempty finite reduced subsets with $CB(m)$, we must have $C=X$.  Hence $\nu_m(X)$ is a circuit.  Theorem~\ref{thm:veronese-circuit-form} gives $|X|\geq md+2$.

For sharpness, the rational-normal-curve examples of Section~\ref{sec:basic-sharpness} give Veronese circuits with
\(
  |X|=m\dim\langle X\rangle+2.
\)
Such a circuit has $CB(m)$, since each Veronese evaluation vector lies in the span of the others.  No proper subset has $CB(m)$, because $CB(m)$ would force a Veronese dependence, whereas every proper subfamily of a circuit is independent.  Hence these examples are inclusion-minimal finite reduced $CB(m)$ sets attaining the bound.
\end{proof}

\section{Proofs of the product and positive-characteristic extensions}\label{sec:product-positive-characteristic}

We prove the product and positive-characteristic extensions stated in Section~\ref{sec:applications}.  The common input is the following multi-factor Schur-product growth principle.

\subsection{A multi-factor Schur-product growth principle}

The following multi-factor growth principle is the technical device behind the Segre--Veronese and positive-characteristic extensions below.

\begin{proposition}\label{prop:multi-factor-schur}
Let $\K$ be an infinite field, and let $E_1,\ldots,E_M\subseteq\K^t$ be nonzero linear subspaces.  Put
\(
  P=E_1E_2\cdots E_M
\)
for their Schur product.  If $\St(P)=\K\one$, then
\[
  \dim P\geq 1+
  \sum_{j=1}^M(\dim E_j-1).
\]
In particular, if $P$ is a hyperplane with a full-support normal vector, then
\[
  t\geq 2+
  \sum_{j=1}^M(\dim E_j-1).
\]
\end{proposition}

\begin{proof}
For $0\leq j\leq M$, set $P_0=\K\one$ and $P_j=E_1\cdots E_j$ for $j\geq1$.  The partial products to which we apply Theorem~\ref{thm:schur-kneser} are nonzero: if some $P_j$ with $j<M$ were zero, then the final product $P$ would be zero; for $t>1$ this would make $\St(P)=\K^t$, and for $t=1$ the product of nonzero subspaces of $\K$ is nonzero.

If $u\in\St(P_j)$, then
\[
  uP=uP_jE_{j+1}\cdots E_M\subseteq P_jE_{j+1}\cdots E_M=P.
\]
Thus $u\in\St(P)=\K\one$, so every nonzero partial product has trivial stabilizer.  Applying Theorem~\ref{thm:schur-kneser} to $P_{j-1}$ and $E_j$ gives
\[
  \dim P_j\geq \dim P_{j-1}+\dim E_j-1
\]
for every $1\leq j\leq M$.  Summing these inequalities from $P_0=\K\one$ yields the first assertion.  If $P$ is a hyperplane with a full-support normal vector, Lemma~\ref{lem:hyperplane-stabilizer} gives $\St(P)=\K\one$ and $\dim P=t-1$, which gives the second assertion.
\end{proof}

\begin{remark}[Equality in the multi-factor bound]\label{rem:critical-chain}
If equality holds in Proposition~\ref{prop:multi-factor-schur}, then equality holds at every intermediate step.  Indeed, the proof expresses $\dim P-1$ as a sum of $M$ increments, the $j$-th increment being bounded below by $\dim E_j-1$.  If the final sum is minimal, no intermediate increment can be larger.  Thus equality cases in the circuit bounds reduce to critical chains for Schur products, a linear analogue of critical-pair problems in additive combinatorics.
\end{remark}

\subsection{Segre--Veronese product circuits}

The same Schur-product argument proves a mixed version for Segre--Veronese points.  This is not merely a formal restatement of Theorem~\ref{thm:main}; it covers support-minimal relations among partially symmetric rank-one tensors.

\begin{proof}[Proof of Theorem~\ref{thm:segre-veronese}]
By Lemma~\ref{lem:base-change-circuits}, linear dependence and support-minimality are unchanged after extending scalars; the dimensions of the spans $L_q$ are unchanged as well.  Thus, if necessary, we first pass to an infinite extension field so that Theorem~\ref{thm:schur-kneser} applies.

Choose a basis of each $L_q$, and write the coefficient row space of the $t$ forms $\ell_{1,q},\ldots,\ell_{t,q}$ in that basis as a code $C_q\subseteq\K^t$.  Thus $\dim C_q=\dim L_q$.

By Lemma~\ref{lem:rank-schur-power}, applied with $j=m_q$, the coefficient row space of $\ell_{1,q}^{m_q},\ldots,\ell_{t,q}^{m_q}$ is $C_q^{\langle m_q\rangle}$; the stated characteristic hypothesis ensures that the required multinomial coefficients do not vanish.  Choose the induced monomial bases in each symmetric-power factor and the corresponding tensor-product basis in $\bigotimes_q\Sym^{m_q}U_q$.  A coefficient row of $F_1,\ldots,F_t$ is then obtained by choosing one coefficient row from each factor and multiplying these rows coordinatewise.  Therefore the coefficient row space of $F_1,\ldots,F_t$ is
\[
  P=C_1^{\langle m_1\rangle}C_2^{\langle m_2\rangle}\cdots C_s^{\langle m_s\rangle}\subseteq\K^t.
\]
The support-minimal dependence assumption gives $\dim P=t-1$, and the unique dependence has full support.  Hence $P$ is a hyperplane with a full-support normal vector, so $\St(P)=\K\one$ by Lemma~\ref{lem:hyperplane-stabilizer}.

List the spaces $C_q$ with repetitions as
\[
  D_1,D_2,\ldots,D_M,
  \qquad
  M=m_1+\cdots+m_s,
\]
where $C_q$ appears exactly $m_q$ times.  Set
\[
  P_j=D_1D_2\cdots D_j
  \quad(1\leq j\leq M),
  \qquad
  P_0=\K\one.
\]
Thus $P_M=P$.  We claim that $\St(P_j)=\K\one$ for every $1\leq j\leq M$.  Indeed, write $P=P_jQ_j$ for the product $Q_j$ of the remaining factors.  If $uP_j\subseteq P_j$, then
\[
  uP=uP_jQ_j\subseteq P_jQ_j=P,
\]
so $u\in\St(P)=\K\one$.

For $j=1$, we have $\dim P_1=\dim D_1$.  For $2\leq j\leq M$, Theorem~\ref{thm:schur-kneser} applied to $S=P_{j-1}$ and $T=D_j$ gives
\[
  \dim P_j\geq \dim P_{j-1}+\dim D_j-1,
\]
because $P_j=P_{j-1}D_j$ and $\St(P_j)=\K\one$.  Iterating yields
\[
  \dim P
  =\dim P_M
  \geq 1+\sum_{j=1}^M(\dim D_j-1)
  =1+\sum_{q=1}^s m_q(\dim L_q-1).
\]
Since $\dim P=t-1$, the desired inequality follows.
\end{proof}

For $s=1$, Theorem~\ref{thm:segre-veronese} recovers Theorem~\ref{thm:main}.  For $s>1$, it gives the analogous circuit bound for partially symmetric rank-one tensors and multihomogeneous diagonal product circuits.

\subsection{Positive characteristic and digit sums}

The same Schur-product mechanism gives the corresponding replacement in positive characteristic.  Write the base-$p$ expansion of $m$ as
\[
  m=m_0+m_1p+\cdots+m_ap^a,
  \qquad 0\leq m_j<p,
\]
and set
\[
  w_p(m)=m_0+m_1+\cdots+m_a.
\]

\begin{lemma}\label{lem:frobenius-coefficient-space}
Let $\K$ be a perfect field of characteristic $p>0$.  Let
\[
  m=m_0+m_1p+\cdots+m_ap^a,
  \qquad 0\leq m_j<p.
\]
Let $\ell_1,\ldots,\ell_t$ be linear forms whose span has basis $x_1,\ldots,x_r$, and let $C\subseteq\K^t$ be the coefficient row space of the family $\ell_1,\ldots,\ell_t$ in this basis.  For $j\geq0$, let
\[
  C^{[p^j]}=\{(c_1^{p^j},\ldots,c_t^{p^j}):(c_1,\ldots,c_t)\in C\}
\]
denote the coordinatewise $p^j$-Frobenius image of $C$.  Since $\K$ is perfect, the map $\sigma_j\colon a\mapsto a^{p^j}$ is a field automorphism.  Although $\sigma_j$ acts only semilinearly on $\K^t$, its image of a $\K$-subspace is again a $\K$-subspace: if $y=\sigma_j(c)$ with $c\in C$ and $\lambda\in\K$, then $\lambda y=\sigma_j(\sigma_j^{-1}(\lambda)c)\in C^{[p^j]}$, and additivity is immediate.  Hence $C^{[p^j]}$ is a $\K$-linear subspace and $\dim C^{[p^j]}=\dim C$.  Then the coefficient row space of $\ell_1^m,\ldots,\ell_t^m$ is
\[
  \prod_{j=0}^a \left(C^{[p^j]}\right)^{\langle m_j\rangle},
\]
where factors with $m_j=0$ are omitted.
\end{lemma}

\begin{proof}
Write
\(
  \ell_i=\sum_{u=1}^r a_{ui} x_u.
\)
In characteristic $p$,
\[
  \ell_i^{p^j}=\sum_{u=1}^r a_{ui}^{p^j} x_u^{p^j}.
\]
For a multi-index $\alpha^{(j)}=(\alpha_{j,1},\ldots,\alpha_{j,r})$ with $|\alpha^{(j)}|=m_j$, the coefficient row in the expansion of $(\ell_i^{p^j})^{m_j}$ is a nonzero scalar multiple of
\[
  \left(\prod_{u=1}^r (a_{ui}^{p^j})^{\alpha_{j,u}}\right)_{i=1}^t.
\]
The scalar is nonzero because $m_j<p$, so all relevant multinomial coefficients are nonzero in $\K$.  These rows span $\left(C^{[p^j]}\right)^{\langle m_j\rangle}$, by multilinearity of coordinatewise products.

Now multiply over all $j$.  The resulting monomial in $x_1,\ldots,x_r$ has exponent vector $\beta=(\beta_1,\ldots,\beta_r)$ with
\(
  \beta_u=\sum_{j=0}^a \alpha_{j,u}p^j.
\)
Since $|\alpha^{(j)}|=m_j<p$, each $\alpha_{j,u}$ lies in $[0,p-1]$, so no carrying occurs in any coordinate.  Hence the entries $\alpha_{j,u}$ are precisely the base-$p$ digits of $\beta_u$, and the exponent vector $\beta$ uniquely recovers all the multi-indices $\alpha^{(j)}$.  Therefore two different choices of the multi-indices cannot contribute to the same monomial.  It follows that the coefficient rows of $\ell_i^m$ are exactly the coordinatewise products of coefficient rows coming from the factors $\left(C^{[p^j]}\right)^{\langle m_j\rangle}$.  Taking spans gives the displayed equality.
\end{proof}

\begin{proof}[Proof of Theorem~\ref{thm:frobenius-corrected}]
It is enough to prove the result after extending scalars to an algebraic closure of $\K$: by Lemma~\ref{lem:base-change-circuits}, minimal linear dependence of a finite family is unchanged by field extension, and the dimension of the span $L$ is unchanged as well.  Thus we may assume that $\K$ is algebraically closed, hence infinite and perfect.

Let $r=\dim L$.  Choose a basis of $L$, and let $C\subseteq\K^t$ be the row space of the coefficient matrix of $\ell_1,\ldots,\ell_t$ in this basis.  By the perfect-field observation in Lemma~\ref{lem:frobenius-coefficient-space}, each coordinatewise Frobenius image $C^{[p^j]}$ is a $\K$-linear subspace of dimension $r$.  By the same lemma, the coefficient row space of $\ell_1^m,\ldots,\ell_t^m$ is
\[
  P=\prod_{j=0}^a \left(C^{[p^j]}\right)^{\langle m_j\rangle},
\]
with factors of multiplicity zero omitted.

By support-minimal dependence, $P$ is a hyperplane in $\K^t$ with a full-support normal vector.  Lemma~\ref{lem:hyperplane-stabilizer} gives $\St(P)=\K\one$.  List the spaces $C^{[p^j]}$ with multiplicity $m_j$; there are $w_p(m)$ factors, each of dimension $r$.  Proposition~\ref{prop:multi-factor-schur} gives
\[
  \dim P\geq 1+w_p(m)(r-1).
\]
Since $\dim P=t-1$, we obtain
\[
  t-1\geq 1+w_p(m)(r-1),
\]
which is equivalent to the stated bound.
\end{proof}

When $p>m$, one has $w_p(m)=m$, so Theorem~\ref{thm:frobenius-corrected} recovers the numerical bound of Theorem~\ref{thm:main}.  We record the corresponding projective form for reference.

\begin{proof}[Proof of Corollary~\ref{cor:veronese-circuit-large-characteristic}]
Choose homogeneous representatives $p_i=[\ell_i]$.  Then
\[
  d+1=\dim\Span\{\ell_1,\ldots,\ell_t\}.
\]
Since $p>m$, Theorem~\ref{thm:frobenius-corrected} gives
\[
  d+1\leq \frac{t+m-2}{m},
\]
which is equivalent to the stated inequality.
\end{proof}

When $m=p^a$, one has $w_p(m)=1$, so Theorem~\ref{thm:frobenius-corrected} gives the ordinary linear circuit bound $\dim L\leq t-1$.  This is sharp, as the next example shows.

\begin{proof}[Proof of Proposition~\ref{prop:frobenius-degeneration}]
For $r=1$, take two equal nonzero linear forms.  For $r\geq2$, take
\[
  \ell_1=x_1,\ldots,\ell_r=x_r,
  \qquad
  \ell_{r+1}=x_1+\cdots+x_r.
\]
In characteristic $p$,
\[
  \ell_{r+1}^{p^a}=x_1^{p^a}+\cdots+x_r^{p^a},
\]
so the $p^a$-th powers are linearly dependent.  They are support-minimally dependent: the vectors $x_1^{p^a},\ldots,x_r^{p^a}$ are linearly independent, and if a proper subfamily contains $\ell_{r+1}^{p^a}$ while omitting, say, $x_j^{p^a}$, then the coefficient of the monomial $x_j^{p^a}$ forces the coefficient of $\ell_{r+1}^{p^a}$ in any relation to be zero, after which all remaining coefficients are zero.  Hence the powers form a circuit.
\end{proof}

\section{Proofs of the extremal and near-extremal results}\label{sec:extremal-stability}

We prove the equality and near-equality statements from Section~\ref{sec:applications}.  The rational-normal-curve construction proves sharpness, but it does not exhaust the equality locus.  The correct language is Hilbert functions and Cayley--Bacharach duality: equality forces the first difference of the Hilbert function to be $(1,a,\ldots,a,1)$, with $m$ copies of $a$, and near-equality allows only a controlled surplus before degree $m$.

Throughout this section, $\K$ has characteristic zero, $m\geq1$, and $a\geq1$.  Let
\[
  X=\{p_1,\ldots,p_t\}\subseteq\PP^a(\K)
\]
be a reduced nondegenerate finite set of $\K$-rational points, meaning $\langle X\rangle=\PP^a$.  We write $H_X(j)$ for its Hilbert function and
\[
  \Delta H_X(j)=H_X(j)-H_X(j-1),
  \qquad H_X(-1)=0,
\]
for its first difference.

We call $X$ an \emph{extremal Veronese circuit configuration} if $\nu_m(X)$ is a circuit and $|X|=ma+2$.

Recall from Definition~\ref{def:cbm} that a finite reduced set $X$ has $CB(m)$ if and only if
\[
  H_{X\setminus\{p\}}(m)=H_X(m)
  \qquad\text{for every }p\in X.
\]

\begin{lemma}\label{lem:circuit-cbp}
If $\nu_m(X)$ is a circuit, then $X$ has $CB(m)$ and
\(
  H_X(m)=|X|-1.
\)
\end{lemma}

\begin{proof}
The value $H_X(m)$ is the linear rank of $\nu_m(X)$.  Since $\nu_m(X)$ is a circuit, this rank is $|X|-1$.  After deleting any point $p\in X$, the remaining Veronese images are linearly independent, so
\[
  H_{X\setminus\{p\}}(m)=|X|-1=H_X(m).
\]
This is exactly $CB(m)$.
\end{proof}

\begin{lemma}\label{lem:evaluation-growth}
Let $\K$ be a field of characteristic zero, and let $X\subseteq\PP^a(\K)$ be a reduced nondegenerate finite set of $\K$-rational points.  If $\nu_m(X)$ is a circuit, then
\[
  H_X(j+1)\geq H_X(j)+a
  \qquad(0\leq j<m).
\]
\end{lemma}

\begin{proof}
Choose homogeneous representatives for the points of $X$, and let $E_j\subseteq\K^X$ be the degree-$j$ evaluation space.  Since $X$ is nondegenerate, $H_X(0)=1$ and $H_X(1)=a+1$, so the assertion for $j=0$ is equality.

Let $C=E_1$.  The degree-$j$ evaluation space is $E_j=C^{\langle j\rangle}$ for every $j\geq0$, because degree-$j$ forms are spanned by products of $j$ linear forms.  Since $\nu_m(X)$ is a circuit, there is a vector $\lambda\in(\K^\times)^X$ such that
\(
  E_m=\lambda^\perp.
\)

By Lemma~\ref{lem:hyperplane-stabilizer}, $\St(E_m)=\K\one$.  If $1\leq q<m$ and $u\in\St(E_q)$, then
\[
  uE_m=uE_qE_{m-q}\subseteq E_qE_{m-q}=E_m,
\]
so $u\in\St(E_m)=\K\one$.  Thus $\St(E_q)=\K\one$ for all $1\leq q\leq m$.

For $1\leq j<m$, apply Theorem~\ref{thm:schur-kneser} to $S=C$ and $T=E_j$.  Since $CE_j=E_{j+1}$ and $\St(E_{j+1})=\K\one$, we obtain
\[
  H_X(j+1)=\dim E_{j+1}\geq \dim C+\dim E_j-1=(a+1)+H_X(j)-1.
\]
This is the desired inequality.
\end{proof}

We shall also use the following elementary property of Hilbert functions of finite reduced sets.

\begin{lemma}\label{lem:no-plateau}
Let $\K$ be an infinite field, and let $X\subseteq\PP^a(\K)$ be a finite reduced set of $\K$-rational points.  If $H_X(d)<|X|$, then
\[
  H_X(d+1)>H_X(d).
\]
\end{lemma}

\begin{proof}
Choose a linear form $h$ that does not vanish at any point of $X$.  Such an $h$ exists because the field is infinite.  Multiplication by $h$ is injective on the homogeneous coordinate ring $R$ of $X$, and the Artinian reduction $A=R/(h)$ satisfies
\[
  \dim_\K A_e=H_X(e)-H_X(e-1)
  \qquad(e\geq0),
\]
with the convention $H_X(-1)=0$.  Indeed, this follows from the exact sequence
\[
  0\longrightarrow R(-1)\xrightarrow{\cdot h}R\longrightarrow A\longrightarrow0.
\]
If $H_X(d+1)=H_X(d)$, then $A_{d+1}=0$.  Since $A$ is a standard graded algebra generated in degree one, this implies $A_e=0$ for every $e\geq d+1$.  Hence $H_X(e)=H_X(d)$ for every $e\geq d$.  This contradicts the fact that the Hilbert function of a finite reduced set eventually equals $|X|$, unless $H_X(d)=|X|$.  Therefore $H_X(d)<|X|$ forces $H_X(d+1)>H_X(d)$.
\end{proof}

\subsection{Hilbert-function classification of equality}

The key point is that equality in the main bound forces equality at every intermediate Schur-power growth step.

\begin{proof}[Proof of Theorem~\ref{thm:extremal-cb}]
Assume first that $\nu_m(X)$ is a circuit and $|X|=ma+2$.  By Lemma~\ref{lem:circuit-cbp}, $X$ has $CB(m)$ and
\[
  H_X(m)=|X|-1=ma+1.
\]
Let $C\subseteq\K^{|X|}$ be the degree-one evaluation code of $X$, after choosing homogeneous representatives.  Since $X\subseteq\PP^a(\K)$ is nondegenerate, $\dim C=a+1$.  The degree-$j$ evaluation space is $C^{\langle j\rangle}$, so
\[
  H_X(j)=\dim C^{\langle j\rangle}.
\]
Lemma~\ref{lem:evaluation-growth} gives the growth inequalities
\[
  H_X(j+1)\geq H_X(j)+a
  \qquad(0\leq j<m).
\]
Starting from $H_X(0)=1$, we get
\[
  H_X(j)\geq ja+1
  \qquad(0\leq j\leq m).
\]
Since equality holds at $j=m$, every intermediate growth step must be an equality.  Hence
\[
  H_X(j)=ja+1
  \qquad(0\leq j\leq m).
\]
By Lemma~\ref{lem:no-plateau}, $H_X(m+1)=|X|$, because $H_X(m)=|X|-1$.  Thus
\[
  \Delta H_X=(1,\underbrace{a,\ldots,a}_{m\text{ times}},1).
\]

Conversely, suppose $X$ has $CB(m)$ and $\Delta H_X=(1,\underbrace{a,\ldots,a}_{m\text{ times}},1)$.  Then
\[
  |X|=ma+2
  \qquad\text{and}\qquad
  H_X(m)=ma+1=|X|-1.
\]
Thus $\nu_m(X)$ is linearly dependent.  Since $X$ has $CB(m)$, for every $p\in X$ we have
\[
  H_{X\setminus\{p\}}(m)=H_X(m)=|X|-1=|X\setminus\{p\}|.
\]
Therefore each deletion $\nu_m(X\setminus\{p\})$ is linearly independent.  Every proper subset is contained in such a deletion, hence is linearly independent.  Thus $\nu_m(X)$ is a circuit.
\end{proof}

We next explain the arithmetically Gorenstein interpretation.  Theorem~\ref{thm:extremal-cb} is self-contained and holds over any field of characteristic zero.

For this interpretation, however, we assume that the ground field is algebraically closed.  We use the criterion of Davis--Geramita--Orecchia \cite[Theorem~5]{DGO}, in the reduced projective zero-dimensional case, together with the modern Cayley--Bacharach formulations of Geramita--Kreuzer--Robbiano \cite{GKR} and Kreuzer--Long--Robbiano \cite{KLR}; see also Eisenbud--Green--Harris \cite{EGH}.  We use only the following standard form of the criterion.  Let
\[
  h_X(j)=\Delta H_X(j),\qquad
  s=\max\{j:h_X(j)\neq0\}.
\]
For a reduced finite set of points, the criterion says that $X$ is arithmetically Gorenstein if and only if the $h$-vector $h_X$ is symmetric and $X$ has the Cayley--Bacharach property in degree $s-1$, namely
\[
  H_{X\setminus\{p\}}(s-1)=H_X(s-1)\qquad(p\in X).
\]
For the $h$-vector $(1,a,\ldots,a,1)$, with $m$ copies of $a$, one has $s=m+1$, so this is exactly the condition $CB(m)$ used in Definition~\ref{def:cbm}.  The elementary deletion formula in Lemma~\ref{lem:cbp-to-dgo} records this matching in the present notation.  Over a general field of characteristic zero, the interpretation may be read after base change to the algebraic closure.

\begin{lemma}\label{lem:cbp-to-dgo}
Let $\K$ be an infinite field, and let $X\subseteq\PP^a(\K)$ be a reduced finite set of $\K$-rational points.  Suppose that
\[
  X\text{ has }CB(m)
  \qquad\text{and}\qquad
  \Delta H_X=(1,\underbrace{a,\ldots,a}_{m\text{ times}},1).
\]
Then $X$ has $CB(e)$ for every $0\leq e\leq m$.  Moreover, for every $p\in X$ and every $e\geq0$,
\[
  H_{X\setminus\{p\}}(e)=\min\{H_X(e), |X|-1\},
\]
and the Hilbert function satisfies the symmetry relation
\[
  H_X(i)+H_X(m-i)=|X|
  \qquad(0\leq i\leq m).
\]
\end{lemma}

\begin{proof}
Let $N=|X|=ma+2$.  The displayed $h$-vector gives
\[
  H_X(i)=1+ia\quad(0\leq i\leq m),
  \qquad H_X(m)=N-1,
  \qquad H_X(e)=N\quad(e\geq m+1).
\]
First let $0\leq e\leq m$.  If $F$ were a degree-$e$ form vanishing on $X\setminus\{p\}$ but not on $p$, choose a linear form $h$ with $h(p)\neq0$.  Then $Fh^{m-e}$ would be a degree-$m$ form vanishing on $X\setminus\{p\}$ but not on $p$, contradicting $CB(m)$.  Thus $X$ has $CB(e)$ for every $e\leq m$.

For $e\leq m$, the equality $H_{X\setminus\{p\}}(e)=H_X(e)$ follows from $CB(e)$, and here $H_X(e)\leq H_X(m)=N-1$.  For $e\geq m$, the $CB(m)$ hypothesis gives
\[
  H_{X\setminus\{p\}}(m)=H_X(m)=N-1.
\]
Since $X\setminus\{p\}$ has $N-1$ points and Hilbert functions are nondecreasing, $H_{X\setminus\{p\}}(e)=N-1$ for all $e\geq m$.  This proves the displayed deletion formula.

Finally, for $0\leq i\leq m$,
\[
  H_X(i)+H_X(m-i)=(1+ia)+(1+(m-i)a)=ma+2=N=|X|.
\]
\end{proof}

\begin{proof}[Justification of Remark~\ref{rem:ag-interpretation}]
The equivalence of \textup{(i)} and \textup{(ii)} is Theorem~\ref{thm:extremal-cb}.  Assume \textup{(ii)}.  The displayed $h$-vector $(1,a,\ldots,a,1)$, with $m$ copies of $a$, is symmetric and has last nonzero degree $s=m+1$.  Moreover, $CB(m)$ is precisely the Cayley--Bacharach condition in degree $s-1$ required by the criterion of Davis--Geramita--Orecchia \cite[Theorem~5]{DGO}, in its reduced zero-dimensional form, and by the equivalent modern formulations of Geramita--Kreuzer--Robbiano \cite{GKR} and Kreuzer--Long--Robbiano \cite{KLR}; Lemma~\ref{lem:cbp-to-dgo} records the deletion formula in our notation.  Hence $X$ is arithmetically Gorenstein with $h$-vector $(1,a,\ldots,a,1)$, with $m$ copies of $a$.

Conversely, if $X$ is arithmetically Gorenstein with this $h$-vector, then $s=m+1$ and the same criterion gives the Cayley--Bacharach property in degree $s-1=m$.  Since $H_X(m)=|X|-1$, condition \textup{(ii)} follows.
\end{proof}

\subsection{Hilbert-function stability}

We now treat near-equality.  Suppose $\nu_m(X)$ is a circuit and write
\[
  |X|=ma+2+s,
  \qquad s\geq0.
\]
Since $H_X(m)=|X|-1$, we have
\(
  H_X(m)=ma+1+s.
\)
Define
\[
  \delta_j=H_X(j)-(ja+1)
  \qquad(0\leq j\leq m).
\]

\begin{proof}[Proof of Theorem~\ref{thm:hilbert-stability}]
We have $\delta_0=0$ because $H_X(0)=1$, and $\delta_1=0$ because $X\subseteq\PP^a(\K)$ is nondegenerate, so $H_X(1)=a+1$.  Lemma~\ref{lem:evaluation-growth} gives
\[
  H_X(j+1)\geq H_X(j)+a
  \qquad(0\leq j<m).
\]
Therefore $\delta_{j+1}\geq\delta_j$.  Finally,
\[
  \delta_m=H_X(m)-(ma+1)=|X|-1-(ma+1)=s.
\]
This proves the monotonicity statement.

If $m=1$, then $s=\delta_m=\delta_1=0$.  Since $H_X(1)=|X|-1$, Lemma~\ref{lem:no-plateau} gives $H_X(2)=|X|$, and hence the displayed first-difference statement reduces to $\Delta H_X=(1,a,1)$.  We may therefore assume $m\geq2$ for the rest of this paragraph.

Put $\varepsilon_j=\delta_j-\delta_{j-1}$ for $2\leq j\leq m$.  Then $\varepsilon_j\geq0$ and
\(
  \sum_{j=2}^m\varepsilon_j=\delta_m-\delta_1=s.
\)
Moreover,
\[
  \Delta H_X(j)=H_X(j)-H_X(j-1)=a+\delta_j-\delta_{j-1}=a+\varepsilon_j
\]
for $2\leq j\leq m$, while $\Delta H_X(0)=1$ and $\Delta H_X(1)=a$.  Since $H_X(m)=|X|-1$, Lemma~\ref{lem:no-plateau} gives $H_X(m+1)=|X|$, and so $\Delta H_X(m+1)=1$.  This proves the displayed shape of the first differences.
\end{proof}

Thus, at the level of Hilbert functions, a near-extremal circuit is obtained from the extremal lower envelope by adding a total surplus of $s$ before degree $m$.

\subsection{Duality stability}

The circuit relation also gives a self-contained duality constraint.  Let
\(
  E_j\subseteq\K^X
\)
be the degree-$j$ evaluation space.  Since $\nu_m(X)$ is a circuit, there is a vector
\(
  \lambda=(\lambda_x)_{x\in X}\in(\K^\times)^X
\)
such that
\[
  E_m=\lambda^\perp
  =\left\{z\in\K^X:\sum_{x\in X}\lambda_x z_x=0\right\}.
\]
Define a perfect pairing on $\K^X$ by
\[
  \langle u,v\rangle_\lambda=
  \sum_{x\in X}\lambda_x u_x v_x.
\]

\begin{proof}[Proof of Theorem~\ref{thm:duality-stability}]
For $f\in E_i$ and $g\in E_{m-i}$, the coordinatewise product $fg$ belongs to $E_m$.  Since $E_m=\lambda^\perp$, we have
\(
  \langle f,g\rangle_\lambda=0.
\)
Thus
\(
  E_{m-i}\subseteq E_i^\perp.
\)
Because the pairing is perfect, $\dim E_i^\perp=|X|-\dim E_i=|X|-H_X(i)$.  Hence
\[
  H_X(m-i)=\dim E_{m-i}\leq |X|-H_X(i),
\]
which proves
\[
  H_X(i)+H_X(m-i)\leq |X|.
\]
Using
\(
  H_X(j)=ja+1+\delta_j
\)
and $|X|=ma+2+s$, this inequality is equivalent to
\(
  \delta_i+\delta_{m-i}\leq s.
\)
The quotient $D_i=E_i^\perp/E_{m-i}$ has dimension
\[
  \dim E_i^\perp-\dim E_{m-i}
  =|X|-H_X(i)-H_X(m-i)
  =s-\delta_i-\delta_{m-i}.
\]
Summing over $0\leq i\leq m$ gives the displayed bound.
\end{proof}

\begin{remark}[Canonical-module interpretation]
In the standard Cayley--Bacharach theory, the condition $CB(m)$ can be expressed through the canonical module of the homogeneous coordinate ring $R=\K[x_0,\ldots,x_a]/I_X$; see Geramita--Kreuzer--Robbiano \cite{GKR} and Kreuzer--Long--Robbiano \cite{KLR}.  Under this interpretation, the quotients $D_i$ above are the graded pieces of the cokernel of the corresponding canonical self-duality map, up to the conventional grading shift.  Thus Theorem~\ref{thm:duality-stability} says that a near-extremal Veronese circuit is almost self-dual, with total duality defect at most $(m+1)s$.
\end{remark}

\subsection{The first near-extremal case}

The first possible nonzero surplus is completely rigid at the level of Hilbert functions.

\begin{proof}[Proof of Corollary~\ref{cor:s-one}]
Here $s=1$, so Theorem~\ref{thm:hilbert-stability} gives
\[
  0=\delta_0=\delta_1\leq\delta_2\leq\cdots\leq\delta_m=1.
\]
Thus each $\delta_j$ is either $0$ or $1$, and there is a unique jump position $k$ such that $\delta_j=0$ for $j<k$ and $\delta_j=1$ for $j\geq k$.

The duality constraint from Theorem~\ref{thm:duality-stability} gives
\(
  \delta_i+\delta_{m-i}\leq1
\)
for every $i$.  If $2k\leq m$, then $i=k$ satisfies both $i\geq k$ and $m-i\geq k$, so
\(
  \delta_i+\delta_{m-i}=2,
\)
a contradiction.  Hence $2k>m$, i.e. $k>m/2$.
\end{proof}


\section*{Acknowledgements}

H.L. was supported by the National Natural Science Foundation of China (12501487), the China Scholarship Council, and IBS-R029-C4.  X.L. was supported by the Excellent Young Talents Program (Overseas) of the National Natural Science Foundation of China.

\section*{Declaration on the use of AI}

The authors used generative AI tools to assist in discussing proof strategies, checking proofs, and improving exposition. All mathematical arguments, results, and conclusions were reviewed and verified by the authors.


\begin{thebibliography}{10}

\bibitem{AlexanderHirschowitz}
J. Alexander and A. Hirschowitz.
\newblock Polynomial interpolation in several variables.
\newblock {\em J. Algebraic Geom.}, 4(2):201--222, 1995.

\bibitem{BernardiEtAl}
A. Bernardi, E. Carlini, M.~V. Catalisano, A. Gimigliano, and A. Oneto.
\newblock The Hitchhiker Guide to: Secant Varieties and Tensor Decomposition.
\newblock {\em Mathematics}, 6(12):Paper No. 314, 86 pp., 2018.

\bibitem{BialynickiBirulaSchinzel}
A. Bia{\l}ynicki-Birula and A. Schinzel.
\newblock Representations of multivariate polynomials by sums of univariate
  polynomials in linear forms.
\newblock {\em Colloq. Math.}, 112(2):201--233, 2008. Corrigendum: Colloq.
  Math. 125 (2011), no. 1, 139.

\bibitem{BukhProblems}
B. Bukh.
\newblock Interesting problems that I cannot solve.
\newblock \url{https://www.borisbukh.org/problems.html}. Accessed June 23,
  2026.

\bibitem{BukhBicliques}
B. Bukh.
\newblock Extremal graphs without exponentially small bicliques.
\newblock {\em Duke Math. J.}, 173(11):2039--2062, 2024.

\bibitem{DGO}
E.~D. Davis, A.~V. Geramita, and F. Orecchia.
\newblock Gorenstein algebras and the {C}ayley-{B}acharach theorem.
\newblock {\em Proc. Amer. Math. Soc.}, 93(4):593--597, 1985.

\bibitem{EGH}
D. Eisenbud, M. Green, and J. Harris.
\newblock Cayley-{B}acharach theorems and conjectures.
\newblock {\em Bull. Amer. Math. Soc. (N.S.)}, 33(3):295--324, 1996.

\bibitem{GKR}
A.~V. Geramita, M. Kreuzer, and L. Robbiano.
\newblock Cayley-{B}acharach schemes and their canonical modules.
\newblock {\em Trans. Amer. Math. Soc.}, 339(1):163--189, 1993.

\bibitem{IarrobinoKanev}
A. Iarrobino and V. Kanev.
\newblock {\em Power sums, {G}orenstein algebras, and determinantal loci},
  volume 1721 of {\em Lecture Notes in Math.}
\newblock Springer-Verlag, Berlin, 1999.

\bibitem{Kneser}
M. Kneser.
\newblock Absch{\"a}tzung der asymptotischen {D}ichte von {S}ummenmengen.
\newblock {\em Math. Z.}, 58:459--484, 1953.

\bibitem{KLR}
M. Kreuzer, L.~N. Long, and L. Robbiano.
\newblock On the {C}ayley-{B}acharach property.
\newblock {\em Comm. Algebra}, 47(1):328--354, 2019.

\bibitem{Landsberg}
J.~M. Landsberg.
\newblock {\em Tensors: geometry and applications}, volume 128 of {\em Grad.
  Stud. Math.}
\newblock American Mathematical Society, Providence, RI, 2012.

\bibitem{MirandolaZemor}
D. Mirandola and G. Z{\'e}mor.
\newblock Critical pairs for the product {S}ingleton bound.
\newblock {\em IEEE Trans. Inform. Theory}, 61(9):4928--4937, 2015.

\bibitem{Randriambololona}
H. Randriambololona.
\newblock On products and powers of linear codes under componentwise
  multiplication.
\newblock In S. Ballet, M. Perret, and A. Zaytsev, editors, {\em Algorithmic
  arithmetic, geometry, and coding theory}, volume 637 of {\em Contemp. Math.},
  pages 3--78. American Mathematical Society, Providence, RI, 2015.

\bibitem{ReznickEvenPowers}
B. Reznick.
\newblock Sums of even powers of real linear forms.
\newblock {\em Mem. Amer. Math. Soc.}, 96(463):viii+155, 1992.

\bibitem{ReznickPatterns}
B. Reznick.
\newblock Patterns of dependence among powers of polynomials.
\newblock In S. Basu and L. Gonzalez-Vega, editors, {\em Algorithmic and
  quantitative real algebraic geometry}, volume 60 of {\em DIMACS Ser.
  Discrete Math. Theoret. Comput. Sci.}, pages 101--121. American Mathematical
  Society, Providence, RI, 2003.

\bibitem{Sladek}
A. S{\l}adek.
\newblock Linear dependence of powers of linear forms.
\newblock {\em Ann. Math. Silesianae}, 29:131--138, 2015.

\end{thebibliography}
\end{document}